\documentclass[12pt]{article}
\usepackage[top=1in, bottom=1.5in, left=1in, right=1in]{geometry}
\usepackage{hyperref}
\usepackage{enumerate}
\usepackage{amsmath,amssymb,amsthm,amsfonts}
\usepackage{stackengine} 
\newcommand\oast{\stackMath\mathbin{\stackinset{c}{0ex}{c}{0ex}{\ast}{\bigcirc}}}

\usepackage{graphicx}
\usepackage{tikz}
\usetikzlibrary{graphs, graphs.standard}

\tikzset{
	modal/.style={>=stealth’,shorten >=1pt,shorten <=1pt,auto,node distance=1.5cm,
		semithick},
	world/.style={circle, draw,minimum size=.1cm,fill=gray!15},
	point/.style={circle,draw,inner sep=0.3mm,fill=black},
	circ/.style={circle,draw,inner sep=0.1mm,fill=white},
	reflexive above/.style={->,loop,looseness=7,in=120,out=60},
	reflexive below/.style={->,loop,looseness=7,in=240,out=300},
	reflexive left/.style={->,loop,looseness=7,in=150,out=210},
	reflexive right/.style={->,loop,looseness=7,in=30,out=330}
}

\usetikzlibrary{shapes}
\usetikzlibrary{plotmarks}
\usetikzlibrary{arrows}
\usetikzlibrary{positioning}
\theoremstyle{definition}
\newtheorem{defn}{Definition}[section]

\newtheorem{thm}[defn]{Theorem}

\newtheorem{conjecture}[defn]{Conjecture}
\newtheorem{lem}[defn]{Lemma}

\newtheorem{claim}{Claim}

\title{Adjacent vertex distinguishing total chromatic number of graph products}

 \author{Amitayu Banerjee$^{1}$, J. Geetha$^2$  and K. Somasundaram$^{2*}$}
 \date{}
\date{\small{$^1$Department of Logic, E\"otv\"os Lor\'and University, Budapest, Hungary.}\\
\small{$^2$Department of Mathematics, Amrita School of Physical Sciences, Coimbatore \\ Amrita Vishwa Vidyapeetham, India.\\
 banerjee.amitayu@gmail.com,\{j\_geetha, s\_sundaram\}@cb.amrita.edu}\\
 $^*$Corresonding author.}
\begin{document}
\maketitle

\begin{abstract}
The {\em AVD-total chromatic number} $\chi''_{a}(G)$ of a graph $G$ is the least integer $k$ for which $G$ has a proper total coloring $f$ with $k$ colors such that $C_G(u)\neq C_G(v)$ for every edge $uv\in E(G)$, where $C_G(u)=\{f(u)\}\cup\{f(uw):uw\in E(G)\}$.
The {\em AVD-total coloring conjecture (AVD-TCC)} asserts that $\chi''_{a}(G)\leq \Delta(G)+3$ for every graph $G$, where $\Delta(G)$ is the maximum degree of $G$.
In this paper, we prove the AVD-TCC for certain classes of Cartesian products, lexicographic products, skew products, cover products, comb products, and Indu-Bala products.
\end{abstract}


\section{Introduction}

All graphs in this paper are finite, simple, and undirected. 
Let $G$ be a graph with vertex set $V(G)$ and edge set $E(G)$.
A \emph{proper total coloring} of $G$ is a coloring of $V(G)\cup E(G)$ such that any two
elements that are either adjacent or incident are assigned different colors.
The minimum number of colors needed for a proper total coloring of $G$ is called the \emph{total chromatic number} of $G$, denoted by $\chi''(G)$. It is known that $\chi''(G) \geq \Delta(G)+1$, where $\Delta(G)$ is the maximum degree of $G$. 
The following conjecture was independently posed by Behzad \cite{Beh1965} and Vizing \cite{Viz1968}. 

\begin{conjecture}[Total Coloring Conjecture (TCC)]\label{Conjecture 1.1}
For any graph $G$, $\chi''(G) \leq \Delta(G) + 2$.
\end{conjecture}

The {\em chromatic index} of $G$, denoted by $\chi'(G)$, is the least integer $k$ such that $G$ admits a proper edge coloring with $k$ colors.
If $\chi'(G)=\Delta(G)$ then $G$ is called Class 1 and if $\chi'(G)=\Delta(G)+1$ then $G$ is called Class 2. We say
$G$ is
of Type $j$ if $\chi''(G) = \Delta(G) + j$, where $j\in\{1,2\}$. If $G$ satisfies the TCC, then $G$ is called a total colorable graph.
A survey on TCC is given in \cite{GNS2023}.

Given a proper total coloring $f$ of $G$, for any vertex $u$ of $G$, we define 
\begin{enumerate}
    \item $C_G(u)=\{f(u)\}\cup\{f(uw):uw\in E(G)\}$, and 
    \item $D_G(u)=\{f(uw):uw\in E(G)\}$. 
\end{enumerate}
An {\em adjacent vertex distinguishing} (AVD)  {\em total coloring} of $G$ is a proper total coloring $f$ such that $C_G(u)\neq C_G(v)$ for every edge $uv\in E(G)$. The {\em AVD-total chromatic number} of $G$, denoted by $\chi''_{a}(G)$, is the minimum number of colors required for an AVD-total coloring of $G$. It is easy to see that $\chi''_{a}(G) \geq \chi''(G)\geq \Delta(G)+1.$ If  
$G$ has two adjacent vertices of maximum degree, then it is known that $\chi''_{a}(G) \geq \Delta(G)+2.$ Similar to the classifications in total coloring, AVD total coloring is classified into three types. If $\chi''_{a}(G) = \Delta(G)+k$, then $G$ is  said to be AVD-Type $k$ for $k=1,2,3$. The following conjecture was posed by Zhang et al. \cite{ZCYLW2005}.

\begin{conjecture}\label{Conjecture 1.2}(AVD-TCC). For any graph $G$, $\chi''_{a}(G) \leq \Delta(G) +3$.
\end{conjecture}

The Conjecture \ref{Conjecture 1.2} is known to hold for $4$-regular graphs~\cite{PR2014}, graphs with $\Delta(G)\in\{3,4\}$~\cite{Wang2007,Che2008,LLLM2017}, complete graphs and bipartite graphs~\cite{ZCYLW2005}, planar graphs with $\Delta(G)\ge 8$~\cite{CHWY2020,CWW2017, HWWY2019,HW2012, WH2014,WHHW2019}, outerplanar graphs~\cite{WW2010}, hypercubes~\cite{CG2009}, complete equipartite graphs~\cite{LCM2015}, indifference graphs~\cite{PM2010}, split graphs~\cite{VFP2022}, and several other graph classes.\\
In this paper, we prove Conjecture \ref{Conjecture 1.2} for certain classes of graph products.

\subsection{Product Graphs}
The four standard graph products are the Cartesian, direct, strong, and lexicographic products.
The total chromatic number of these graph products were studied in \cite{GK2018, KM2003, SAWW1997, SGV2023}. The survey \cite{GNS2023} gives more details on total coloring for product graphs.

In \cite{BKT2012, TW2015, Wang2017}, AVD total colorings for the 
{\em Cartesian product} $G \square H$, the {\em direct product} $G \times H$, the {\em lexicographic product} $G \circ H$, and the {\em strong product} $G \boxtimes H$ were investigated for several classes of graphs.
In this paper, we investigate the AVD-total chromatic number of some new classes of the Cartesian product, the lexicographic product, and the skew product $G \Delta H$ introduced by Shibata and Kikuchi~\cite{SK2000}.

The vertex set of each of these products is
$V(G)\times V(H)$ and the edge sets are defined as follows:
\begin{align*}
E(G \square H)
&=
\Bigl\{
((g,h),(g',h'))
:
\bigl(g=g' \text{, } hh'\in E(H)\bigr)
\text{ or }
\bigl(gg'\in E(G) \text{, } h=h'\bigr)
\Bigr\},
\\[1mm]
E(G\times H)
&=
\Bigl\{
((g,h),(g',h'))
:
gg'\in E(G)
\text{, }
hh'\in E(H)
\Bigr\},
\\[1mm]
E(G \circ H)
&=
\Bigl\{
((g,h),(g',h'))
:
\bigl(g=g' \text{, } hh'\in E(H)\bigr)
\text{ or }
gg'\in E(G)
\Bigr\},
\\[1mm]
E(G \Delta H)
&=
\Bigl\{
((g,h),(g',h'))
:
\bigl(g=g' \text{,} \ hh'\in E(H)\bigr)
\text{ or }
\bigl(gg'\in E(G) \text{ , } \  hh'\in E(H)\bigr)
\Bigr\}.
\end{align*}
We also study the cover product introduced by Llamas and Bernal~\cite{LB2015}, the comb product introduced by Accardi, Ghorbal, and Obata~\cite{AGO2004}, and the Indu--Bala product~\cite{IB2016}.
The \emph{cover product} of two graphs $G$ and $H$ with fixed vertex covers $VC(G)$ and $VC(H)$ is a graph $G\circledast H$ with vertex set
$
V(G)\cup V(H)
$
and edge set
\[
E(G)\cup E(H)\cup
\bigl\{
\{i,j\}: i\in VC(G),\; j\in VC(H)
\bigr\}.
\]
The \emph{comb product} of $G$ and $H$ \cite{AGO2004, SGV2023} with a distinguished vertex $o\in V(H)$ is the graph $G\triangleright_o H$
obtained by grafting a copy of $H$ at vertex $o$ into each vertex of $G$.
In particular, $G\triangleright_o H$ is a graph with vertex set
$
V(G\triangleright_o H)
=
\{
(g,h): g\in V(G),\; h\in V(H)
\},
$
and edge set
\[
E(G\triangleright_o H)
=
\Bigl\{
((g,h),(g',h'))
:
\bigl(gg'\in E(G) \text{ and } h=h'=o\bigr)
\text{ or }
\bigl(g=g' \text{ and } hh'\in E(H)\bigr)
\Bigr\}.
\]
The \emph{Indu--Bala product} of $G$ and $H$~\cite{IB2016, SGV2023}, denoted by $G\blacktriangledown H$, is obtained from two disjoint copies of the join $G\vee H$ by joining the corresponding vertices belonging to the two copies of $H$.

\subsection{Main Results}
\begin{enumerate}
    \item (Theorem \ref{Theorem 3.1}) The Cartesian product $G\square H$ satisfies the AVD-TCC if $G$ is a bipartite graph and $H$ satisfies $\chi^{''}_{a}(H)\leq \Delta(H)+2$.
    
    \item (Theorem \ref{Theorem 4.3}) The lexicographic product $G\circ H$ satisfies the AVD-TCC if $G$ is a balanced bipartite graph with perfect matching and $H$ is a Type 1 graph. 
    
    \item (Theorem \ref{Theorem 5.1}) If $G$ is a Class 1 graph and $H$ is an AVD-Type 1 graph, then the skew product $G\Delta H$ is 
    an AVD-Type 1 graph. 
    Moreover, if $G$ is a Class 1 graph and $\chi''_{a}(H)\leq \Delta(H)+i$ for $i\in \{1,2,3\}$, then $\chi''_{a}(G\Delta H)\leq \Delta(G\Delta H) +i$.

    \item (Theorem \ref{Theorem 6.1}) If $G$ and $H$ are two connected graphs such that $G$ satisfies TCC and $H$ has AVD-Type 2, then the comb product $G \triangleright_o H$ satisfies the AVD-TCC if $n\leq \deg_{H}(o)+1$ where $n$ is the order of $G$ and $o$ is the distinguished vertex of $G\triangleright_o H$.

    \item (Theorem \ref{Theorem 6.2}) If $G$ and $H$ are connected graphs of AVD-Type 2, then the comb product $G \triangleright_o H$ satisfies the AVD-TCC.

    \item (Theorem \ref{Theorem 7.1}) If $G$ and $H$ are two total colorable connected graphs and $k_{1}$, $k_{2}$ are the vertex covering numbers of $G$ and $H$, respectively, such that 
\begin{center}
    $\Delta(H) \leq \Delta(G)$ and $k_{1}+1 < k_{2}$ or $\Delta(G) \leq \Delta(H)$ and
$k_{2}+1 < k_{1}$, 
\end{center}
then the cover
product $G \oast H$ satisfies the AVD-TCC.

    \item (Theorem \ref{Theorem 8.1}) If $G_1$ and $G_2$ are graphs of order $m$ and $n$ respectively such that
    \begin{enumerate}
        \item $m>n+1$,
        \item $\Delta(G_2)\geq\Delta(G_1)$, and
        \item $G_2$ satisfies the TCC,
    \end{enumerate}
    then the Indu-Bala product $G_1\blacktriangledown G_2$ satisfies the AVD-TCC.
\end{enumerate}

\section{Preliminaries} 

A {\em Latin square} of order $k$ is a $k \times k$ array based on the elements $1, 2,...,k$ such that each element occurs exactly once in each row and exactly once in each column. A Latin square $M = [m_{i,j}]$ of order $k$ is said to be {\em commutative} if $m_{i,j} = m_{j,i}$, for $1 \leq i, j \leq k$. If the rows of $M$ are just cyclic permutations (one shift of the elements to the left) of the previous row, then $M$ is said to be {\em anti-circulant}.

\begin{lem}\label{Lemma 2.1}
{\em If $M = [m_{i,j}]$ is a Latin square where $m_{i,j} \equiv (i + j)-1 \pmod{n}, 1 \leq m_{i,j} \leq n$, for $1 \leq i, j \leq n$, then $M$ is an anti-circulant commutative latin square of order $n$.} 
\end{lem}

\begin{lem}\label{Lemma 2.2}
    {\em Fix $n>2$. Let $M = [m_{i,j}]$ be a Latin square where $m_{i,j} \equiv (i + j)-1 \pmod{n}, 1 \leq m_{i,j} \leq n$, for $1 \leq i, j \leq n$. If $k,l<n$, and $M'$ is the submatrix of $M$ consisting of the first $k$ rows and the first $l$ columns of $M$, then 
    \begin{enumerate}
        \item no two rows of $M'$ contain the same set of entries,
        \item no two columns of $M'$ contain the same set of entries.
    \end{enumerate}
    }
\end{lem}

\begin{proof}
(1). 
Let $R_i=\{m_{i,1},...,m_{i,l}\}$ be the $i^{th}$ row of $M'$ for some $1\leq i\leq k$.
For the sake of contradiction, assume that 
$R_p=R_{q}$ where $p> q$ and $1\leq p,q\leq k$.  
For the entry $m_{p,1}$ to appear in $R_q$, there must exist some index $j \in \{1, \dots, l\}$ such that
$m_{p,1}=m_{q,j}$, i.e.,
$
(p+1)-1 \equiv (q + j)-1 \pmod{n}.
$
It follows that
$p+1 \equiv q + j \pmod{n}$,
and hence
$j \equiv p - q + 1 \pmod{n}$.
Since $1 \le p, q \le k<n$, we obtain $1 < p - q + 1 \le k< n$, which implies
$
j = p - q + 1.
$
By Lemma \ref{Lemma 2.1}, $M$ is anticirculant.
Consequently, 
$m_{p,1}=m_{q,j},m_{p,2}=m_{q,j+1}, \dots,\; m_{p,(l+1)-(j-1)}=m_{q,l+1}$. 
We note that $j=p-q+1\ge2$ since $p>q$. 
Thus, $(l+1)-(j-1)\le l$, and hence $m_{q,l+1}\in\{m_{p,1},\dots,m_{p,l}\}=R_{p}$.
However, by the definition of a latin square,
$
m_{q,l+1}\notin\{m_{q,1},\dots,m_{q,l}\}=R_q
$, which contradicts the assumption that $R_p=R_q$.

Similarly, we can prove (2).
\end{proof}

\begin{thm}[K\H{o}nig's Theorem]\label{Theorem 2.3}
{\em If $G$ is a bipartite graph, then $\chi'(G) = \Delta(G)$.}
\end{thm}

\section{Cartesian Product}

\begin{thm}\label{Theorem 3.1}
    {\em Let $G$ be a bipartite graph and $H$ be a graph such that $\chi^{''}_{a}(H)\leq \Delta(H)+2$. Then, $G\square H$ satisfies the AVDTCC.}
\end{thm}

\begin{proof}
Let $G$ be a bipartite graph with bipartitions $X=\{x_1,...,x_n\}$ and $Y=\{y_1,...,y_m\}$
and $H_{v_i}$ be the copy of $H$ with respect to the vertex $v_i\in V(G)$. We note that $\Delta(G\square H)=\Delta(G)+\Delta(H)$. We divide $\Delta(G)+\Delta(H)+3$ colors into the following disjoint sets:
\begin{align*}
C_0 &= \{a_1^0, a_2^0, \dots, a_{\Delta(H)+2}^0\}, \\
C_i &= \{a_i^1\}, 1\leq i\leq \Delta(G),\\
C_{\Delta(G)+1} &= \{a_{\Delta(H)+3}^0\}.
\end{align*}

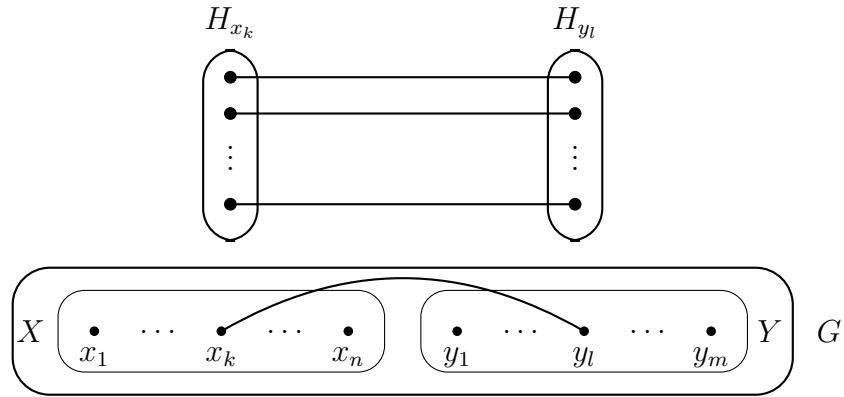
\begin{figure}[ht]
    \centering
    \begin{tikzpicture}[scale=1.2]
        
        \draw[thick, rounded corners=12pt]
            (-2.2,-0.6) rectangle (-1.6,1.5);
        \draw[thick, rounded corners=12pt]
            (1.6,-0.6) rectangle (2.2,1.5);

        \node at (-1.9,1.8) {$H_{x_k}$};
        \node at (1.9,1.8) {$H_{y_l}$};

        \fill (-1.9,1.2) circle (2pt) coordinate (L1);
        \fill (-1.9,0.8) circle (2pt) coordinate (L2);
        \fill (-1.9,-0.2) circle (2pt) coordinate (L3);
        \node at (-1.9,0.4) {$\vdots$};

        \fill (1.9,1.2) circle (2pt) coordinate (R1);
        \fill (1.9,0.8) circle (2pt) coordinate (R2);
        \fill (1.9,-0.2) circle (2pt) coordinate (R3);
        \node at (1.9,0.4) {$\vdots$};

        \draw[thick] (L1) -- (R1);
        \draw[thick] (L2) -- (R2);
        \draw[thick] (L3) -- (R3);

        \begin{scope}[yshift=-2.1cm]

            \draw[thick, rounded corners=14pt]
                (-4.3,-0.2) rectangle (4.3,1.2);
            \node at (4.7,0.5) {$G$};

            \draw[thick]
                (-2.0,0.5) to[out=30,in=150] (2.0,0.5);

            \draw[rounded corners=10pt]
                (-3.8,0.05) rectangle (-0.2,0.95);
            \node at (-4.1,0.5) {$X$};

            \fill (-3.4,0.5) circle (1.5pt);
            \node[below=2pt] at (-3.4,0.5) {$x_1$};

            \node at (-2.7,0.5) {$\ldots$};

            \fill (-2.0,0.5) circle (1.5pt);
            \node[below=2pt] at (-2.0,0.5) {$x_k$};

            \node at (-1.3,0.5) {$\ldots$};

            \fill (-0.6,0.5) circle (1.5pt);
            \node[below=2pt] at (-0.6,0.5) {$x_n$};


            \draw[rounded corners=10pt]
                (0.2,0.05) rectangle (3.8,0.95);
            \node at (4.05,0.5) {$Y$};

            \fill (0.6,0.5) circle (1.5pt);
            \node[below=2pt] at (0.6,0.5) {$y_1$};

            \node at (1.3,0.5) {$\ldots$};

            \fill (2.0,0.5) circle (1.5pt);
            \node[below=2pt] at (2.0,0.5) {$y_l$};

            \node at (2.7,0.5) {$\ldots$};

            \fill (3.4,0.5) circle (1.5pt);
            \node[below=2pt] at (3.4,0.5) {$y_m$};

        \end{scope}

    \end{tikzpicture}
    \caption{Edges between $H_{x_k}$ and $H_{y_l}$ in $G\square H$
    if $x_k y_l\in E(G)$.}
    \label{Figure 1}
\end{figure}

We define a total coloring $c:V(G\Box H)\cup E(G\Box H)\rightarrow \bigcup\limits_{i=0}^{\Delta(G)+1} C_i$ of $G\Box H$ as follows:

\begin{enumerate}
    \item For every $x_i\in X$, we consider an AVD total coloring of $H_{x_i}$ with colors from $C_0$ so that the corresponding elements of $H_{x_i}$ receive the same colors. In particular, let
$
\varphi:V(H)\cup E(H)\rightarrow C_0
$
be an AVD total coloring of \(H\). This is possible since we assumed $\chi''_{a}(H)\leq \Delta(H)+2$. For every \(x_k\in X\), define
$c((x_k,x))=\varphi(x)$ for all $x\in V(H)$,
and
$c((x_k,x)(x_k,y))=\varphi(xy)$ for all $xy\in E(H)$.   


\item By Theorem~\ref{Theorem 2.3}, there exists a proper edge coloring $\theta:E(G)\rightarrow \{1,...,\Delta(G)\}$ of $G$.
    For every edge $x_ky_l\in E(G)$ with $\theta(x_ky_l)=i$, 
    let $\phi_{kl}: E(B(H_{x_k},H_{y_l})) \rightarrow C_i$ be a proper edge coloring of $B(H_{x_k},H_{y_l})$ with the color $a^1_i\in C_i$, where $B(H_{x_k},H_{y_l})$ is the bipartite subgraph graph of $G\square H$ with bipartitions $H_{x_k}$ and $H_{y_l}$ (see Fig. \ref{Figure 1}). Such a coloring exists since the edges of $B(H_{x_k},H_{y_l})$ is a perfect matching and
\begin{center}
$
\chi'(B(H_{x_k},H_{y_l}))
=
\Delta(B(H_{x_k},H_{y_l}))
=
1
=
|C_i|
$
\end{center}
by Theorem~\ref{Theorem 2.3}.
We define $c(e)=\phi_{kl}(e)$ for every edge $e\in E(B(H_{x_k},H_{y_l}))$.


\item Fix $y_i\in Y$. Define a function $s:C_0\rightarrow C_0\backslash \{a^0_1\}\cup C_{\Delta(G)+1}$ by $s(a^0_i)=a^{0}_{i+1}\in  C_0\backslash \{a^0_1\}$ for each $1\leq i\leq \Delta(H)+1$ and $s(a^{0}_{\Delta(H)+2})=a_{\Delta(H)+3}^0\in C_{\Delta(G)+1}$.
Let $\varphi':V(H)\cup E(H)\rightarrow \Delta(H)+2$ be an AVD total coloring of $H$ where 
\begin{center}
$\varphi'(x)=s(\varphi(x))$     
\end{center}
for any $x\in V(H)\cup E(H)$.
For every $y_k\in Y$, define
$c((y_k,x))=\varphi'(x)$ for all $x\in V(H)$,
and
$c((y_k,x)(y_k,y))=\varphi'(xy)$ for all $xy\in E(H)$.  
\end{enumerate}

\begin{claim}\label{Claim 1}
{\em The coloring $c$ is a proper total coloring of $G\square H$.}
\end{claim}

\textit{Proof:}
Every edge joining two distinct copies of $H$ receives a color from
$
\bigcup\limits_{t=1}^{\Delta(G)} C_t,
$
whereas every vertex and every edge within a copy of $H$ receives a color from $C_0\cup C_{\Delta(G)+1}$. Since the color sets $C_1,\ldots,C_{\Delta(G)}$ are disjoint from $C_0\cup C_{\Delta(G)+1}$, every edge joining distinct copies of $H$ receives a color different from every color used on the vertices and edges within those copies.
Furthermore, since $\theta$ is a proper edge coloring of $G$, whenever the edges $v_kv_l$ and $v_rv_s$ are adjacent in $G$, the corresponding bipartite graphs $B(H_{v_k},H_{v_l})$ and $B(H_{v_r},H_{v_s})$ are colored from distinct color sets. Hence, any two adjacent edges in $G\square H$ that join distinct copies of $H$ receive different colors.
Since $s(a^{0}_{i})\neq a^{0}_{i}$ for any $1\leq i\leq \Delta(H)+2$, we have $c(p)\neq c(q)$ if $p\in V(H_{x_i}), q\in V(H_{y_j})$ and $pq\in E(G\square H)$.  
In particular, if $p=(x_{i},z)$ and $q=(y_{j},z)$ for some $z\in H$, then \[c(p)=\varphi(z)\neq s(\varphi(z))=\varphi'(z)=c(q).\]
Moreover, for each edge $v_kv_l\in E(G)$, the coloring
$
\phi_{kl}:E(B(H_{v_k},H_{v_l}))\rightarrow C_{\theta(v_kv_l)}
$
is a proper edge coloring of $B(H_{v_k},H_{v_l})$, and each copy $H_{v_k}$ is properly AVD total colored. Therefore, $c$ is a proper total coloring of $G\square H$. 

\begin{claim}\label{Claim 2}
{\em The coloring $c$ is an AVD-total coloring of $G\square H$.}
\end{claim}

\textit{Proof:}
Pick an edge $xy\in E(G\square H)$. We show that $C_{G\square H}(x)\neq C_{G\square H}(y)$.  

\textbf{Case (i).} $x,y\in H_{x_i}$ for some $x_i$.
Since $H_{x_i}$ is AVD total colored, we obtain $C_{H_{x_{i}}}(x)\neq C_{H_{x_{i}}}(y)$. Furthermore, $C_{H_{x_{i}}}(x), C_{H_{x_{i}}}(y)\subseteq C_0$, $(C_{G\square H}(x)\backslash C_{H_{x_{i}}}(x)) \cap C_0 =\emptyset$ and $(C_{G\square H}(y)\backslash C_{H_{x_{i}}}(y)) \cap C_0 =\emptyset$.
Consequently, $C_{G\square H}(x)\neq C_{G\square H}(y)$.   

\textbf{Case (ii).} $x,y\in H_{y_i}$ for some $y_i$. Similar to Case (i), we can see that 
$C_{G\square H}(x)\neq C_{G\square H}(y)$.

\noindent
\textbf{Case (iii).} Let $x\in H_{x_i}$ and $y\in H_{y_j}$, where $x_iy_j\in E(G)$. Since the colors assigned to the edges connecting different $H_{x}$'s belong to $\bigcup\limits_{i=1}^{\Delta(G)}C_i$, which is disjoint from $C_0\cup C_{\Delta(G)+1}$, it is sufficient to show that
\[
C_{H_{x_i}}(x)\neq C_{H_{y_j}}(y).
\]
For the sake of contradiction, assume that
$
C_{H_{x_i}}(x)=C_{H_{y_j}}(y)
$.
Let $C_{H_{x_i}}(x)=\{a^0_{p_1},...,a^0_{p_l}\}$ 
and 
$C_{H_{y_j}}(y)=\{a^0_{q_1},...,a^0_{q_l}\}$ be the enumerations of $C_{H_{x_i}}(x)$ and $C_{H_{y_j}}(y)$ such that $p_1<p_2<...<p_l$ and $q_1<q_2<...<q_l$. 
If $p_1=1$, then $a^0_1=a^0_{p_1}\in C_{H_{x_i}}(x)$. However, the map $s$ does not map any color of $C_0$ to $a^0_1$. Hence,
$
a^0_1\notin C_{H_{y_j}}(y),
$
contradicting the assumption that
$
C_{H_{x_i}}(x)=C_{H_{y_j}}(y).
$
Therefore, $p_1>1$. 
Since the total coloring of $H_{y_j}$ is obtained by applying the
map $s$, the color $s(a^0_{p_1})$ belongs to
$C_{H_{y_j}}(y)$. In particular,
\[
s(a^0_{p_1})=
\begin{cases}
a^0_{p_1+1},&p_1<\Delta(H)+2,\\
a^0_{\Delta(H)+3},&p_1=\Delta(H)+2.
\end{cases}
\]
Since $C_{H_{x_i}}(x)=C_{H_{y_j}}(y)$, we have either $a^0_{p_1+1}\in C_{H_{x_i}}(x)$ or $a^0_{\Delta(H)+3}\in C_{H_{x_i}}(x)$. The latter case is impossible because 
$C_{H_{x_i}}(x)\subseteq C_0$,
so we assume $a^0_{p_1+1}\in C_{H_{x_i}}(x)$.

\textbf{Subcase (i):} $p_1+1=q_k$ for some $k>1$. Since $a^{0}_{q_1}\in C_{H_{y_j}}(y)$, there must exist some $a^{0}_{p_t}\in C_{H_{x_i}}(x)$ such that $p_t+1=q_1$. Since $q_1<q_k$, we obtain $p_t<p_1$ which contradicts the fact that $p_1<p_i$ for all $i\in\{2,...,l\}$. 
    \vspace{2mm}

\textbf{Subcase (ii):} $p_1+1=q_1$. Since $a^{0}_{q_1}\in C_{H_{y_j}}(y)=C_{H_{x_i}}(x)$, we have $p_2=q_1$.
    If $p_2+1= q_k$ for some $k>2$, then by the arguments similar to Subcase (i) we obtain a contradiction. Thus, $p_2+1=q_2$ and $p_3=q_2$. Continuing in this way, we obtain $p_{i+1}=q_i$ for each $2\leq i\leq l-1$.
    We note that
\[
q_l=
\begin{cases}
p_{l}+1,& p_l<\Delta(H)+2,\\
\Delta(H)+3,& p_l=\Delta(H)+2.
\end{cases}
\]
In either case, $q_l>p_l$, and hence $q_l\neq p_t$ for every $1\le t\le l$, contradicting the assumption
$
C_{H_{x_i}}(x)=C_{H_{y_j}}(y).
$
\end{proof}

\section{Direct and Lexicographic Products}

\begin{thm}\cite[Lemma 1.5]{TW2015}\label{Theorem 4.1}
{\em If $G$ is a bipartite graph, then $\chi''_{a}(G)\leq \Delta(G)+2$.}  
\end{thm}

We know that the direct product $G\times H$ is bipartite if either $G$ or $H$ is a bipartite graph, so Theorem \ref{Theorem 4.2} follows from Theorem \ref{Theorem 4.1}.

\begin{thm}\label{Theorem 4.2}
    {\em Let $G$ be a bipartite graph and $H$ be any connected graph of order $n\geq 3$. Then $\chi''_{a}(G\times H)\leq \Delta(G\times H)+2$.}
\end{thm}

\begin{thm}\label{Theorem 4.3}
    {\em Let $G$ be a balanced bipartite graph with a perfect matching and $H$ be a Type 1 nontrivial graph. Then $G\circ H$ satisfies the AVDTCC.}
\end{thm}

\begin{proof}
We note that $\Delta(G\circ H) =
\Delta(H)+k\Delta(G)$ where $k\geq 2$ is the order of $H$.
Let $V(G)=\{v_1,...,v_m\}$ be an enumeration of the set of vertices of $G$ and $H_{v_i}$ be the copy of $H$ with respect to the vertex $v_i$. 
By assumptions, $m$ is even.
Let 
$X=\{v_i: i\text{ is even}\}$ and $Y=\{v_i: i\text{ is odd}\}$ be the bipartitions of $G$
and
$M=\{v_1v_2, v_3v_4,...,v_{m-1}v_m\}$ be a perfect matching of $G$.  
We divide $\Delta(H) + k\Delta(G) + 3$ colors into the following $\Delta(G) + 2$ disjoint sets:
\begin{align*}
C_0 &= \{a_1^0, a_2^0, \dots, a_{\Delta(H)+1}^0\}, \\
C_i &= \{a_1^i, a_2^i, \dots, a_{k}^i\}, 1\leq i\leq \Delta(G),\\
C_{\Delta(G)+1} &= \{t_0, t_1\}.
\end{align*}

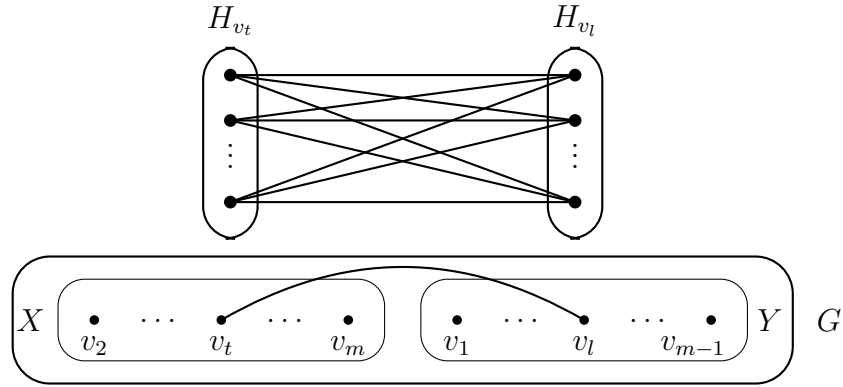
\begin{figure}[ht]
    \centering
    \begin{tikzpicture}[scale=1.2]
        
        \draw[thick, rounded corners=12pt]
            (-2.2,-0.6) rectangle (-1.6,1.5);
        \draw[thick, rounded corners=12pt]
            (1.6,-0.6) rectangle (2.2,1.5);

        \node at (-1.9,1.8) {$H_{v_t}$};
        \node at (1.9,1.8) {$H_{v_l}$};

        \fill (-1.9,1.2) circle (2pt) coordinate (L1);
        \fill (-1.9,0.7) circle (2pt) coordinate (L2);
        \fill (-1.9,-0.2) circle (2pt) coordinate (L3);
        \node at (-1.9,0.4) {$\vdots$};

        \fill (1.9,1.2) circle (2pt) coordinate (R1);
        \fill (1.9,0.7) circle (2pt) coordinate (R2);
        \fill (1.9,-0.2) circle (2pt) coordinate (R3);
        \node at (1.9,0.4) {$\vdots$};

        \draw[thick] (L1) -- (R1);
        \draw[thick] (L1) -- (R2);
        \draw[thick] (L1) -- (R3);
        
        \draw[thick] (L2) -- (R1);
        \draw[thick] (L2) -- (R2);
        \draw[thick] (L2) -- (R3);
        
        \draw[thick] (L3) -- (R1);
        \draw[thick] (L3) -- (R2);
        \draw[thick] (L3) -- (R3);


        \begin{scope}[yshift=-1.5cm]

            \draw[thick, rounded corners=14pt]
                (-4.3,-0.7) rectangle (4.3,0.7);
            \node at (4.7,0) {$G$};

            \draw[thick]
                (-2.0,0) to[out=30,in=150] (2.0,0);


            \draw[rounded corners=10pt]
                (-3.8,-0.45) rectangle (-0.2,0.45);
            \node at (-4.1,0) {$X$};

            \fill (-3.4,0) circle (1.5pt);
            \node[below=2pt] at (-3.4,0) {$v_2$};

            \node at (-2.7,0) {$\ldots$};

            \fill (-2.0,0) circle (1.5pt);
            \node[below=2pt] at (-2.0,0) {$v_t$};

            \node at (-1.3,0) {$\ldots$};

            \fill (-0.6,0) circle (1.5pt);
            \node[below=2pt] at (-0.6,0) {$v_m$};


            \draw[rounded corners=10pt]
                (0.2,-0.45) rectangle (3.8,0.45);
            \node at (4.05,0) {$Y$};

            \fill (0.6,0) circle (1.5pt);
            \node[below=2pt] at (0.6,0) {$v_1$};

            \node at (1.3,0) {$\ldots$};

            \fill (2.0,0) circle (1.5pt);
            \node[below=2pt] at (2.0,0) {$v_l$};

            \node at (2.7,0) {$\ldots$};

            \fill (3.4,0) circle (1.5pt);
            \node[below=2pt] at (3.2,0) {$v_{m-1}$};

        \end{scope}

    \end{tikzpicture}
    \caption{Edges between $H_{v_t}$ and $H_{v_l}$ in $G \circ H$ if $v_t v_l \in E(G)$.}
    \label{Figure 2}
\end{figure}

Let $c:V(G\circ H)\cup E(G\circ H)\rightarrow \bigcup\limits_{i=0}^{\Delta(G)+1} C_i$ be a total coloring of $G\circ H$ defined as follows:

\begin{enumerate}
\item  We consider a proper total coloring of $H_{v_t}$ if $t$ is odd and a proper edge coloring of $H_{v_t}$ if $t$ is even with $\Delta(H)+1$ colors from the color set $C_0$. This is possible since $H_{v_t}$ is a copy of $H$, which is a Type 1 graph.
In particular, let $\varphi:V(H)\cup E(H)\rightarrow C_0$ be a proper total coloring of $H$. 
For every $v_t\in Y$, we define
$c((v_t,x))=\varphi(x)$ for all $x\in V(H)$,
and
$c((v_t,x)(v_t,y))=\varphi(xy)$ for all $xy\in E(H)$. For every $v_t\in X$, we define $c((v_t,x)(v_t,y))=\varphi(xy)$ for all $xy\in E(H)$.

\item Let $\theta:E(G)\rightarrow \{1,...,\Delta(G)\}$ be a proper edge coloring of $G$.
For every edge $v_tv_l\in E(G)\setminus M$ with
$\theta(v_tv_l)=i$, let $\phi_{tl}: E(B(H_{v_t},H_{v_l})) \rightarrow C_i$ be a proper edge coloring of $B(H_{v_t},H_{v_l})$ with the colors from $C_i$, where $B(H_{v_t},H_{v_l})$ is the bipartite subgraph  of $G\circ H$ with bipartitions $H_{v_t}$ and $H_{v_l}$ (see Fig. \ref{Figure 2}). Such a coloring exists since 
\begin{center}
$\chi'(B(H_{v_t},H_{v_l}))=
\Delta(B(H_{v_t},H_{v_l}))=
k=|C_i|$
\end{center}
by Theorem~\ref{Theorem 2.3}.
We define $c(e)=\phi_{tl}(e)$ for every edge $e\in E(B(H_{v_t},H_{v_l}))$.

    \item Suppose $\theta(v_xv_{y})=i$ where $v_xv_y\in M$ and $x<m$ is odd. We note that the vertices of $H_{v_x}$ are already colored in (1). 
    We will color the join edges between $H_{v_x}$ and $H_{v_{y}}$ and the vertices of $H_{v_{y}}$. 
    Consider a latin square $M=[m_{i,j}]$ of order $k+2$  as in Lemma \ref{Lemma 2.2}. 
    We define a bijection $f:\{1,...,k+2\}\rightarrow C_{i}\cup C_{\Delta(G)+1}$ as follows:
\begin{center}
    $f(r)=a_{r}^{i}, 1\leq r\leq k, \qquad f(k+1)=t_0, \qquad f(k+2)=t_1.$
\end{center}
    Let $R$ be a matrix of order $k+2$ whose entries are defined as $r_{i,j}=f(m_{i,j})$ for $1\leq i,j\leq k+2$.
    Let $R'$ be the submatrix of $R$ containing the first $k+1$ rows and the first $k$ columns. 
    Let $H_{v_x}=\{v_x^1,...,v_x^k\}$ and $H_{v_{y}}=\{v_{y}^1,...,v_{y}^k\}$.
    We color the join edges of $H_{v_x}$ and $H_{v_{y}}$ and the vertices of $H_{v_{y}}$ using the entries of $R'$ as follows:
\begin{align*}
c(y)=
\begin{cases}
r_{i,j} & \text{if } y=v_x^{i}v_{y}^j,   1\leq i,j \leq k 
\\[2mm]
r_{k+1,j} & \text{if } y=v_{y}^{j}, 1\leq j \leq k. 
\end{cases}
\end{align*}
\end{enumerate}

\begin{claim}\label{Claim 3}
{\em The coloring $c$ is an AVD-total coloring of $G\circ H$.}
\end{claim}

\textit{Proof.}
By construction, $c$ is a proper total coloring of $G\circ H$.
Let $uv\in E(G\circ H)$. By a careful introspection, we will analyze the following cases:

\textbf{Case (i).}
Let $u=(v_i,a)\in H_{v_i}$ and $v=(v_j,b)\in H_{v_j}$ where $i$ and $j$ have different parity for some $a,b\in V(H)$.
Without loss of generality, assume that $i$ is odd and $j$ is even. Then 
\begin{center}
    $|C_{G\circ H}(v)\cap (\bigcup\limits_{t=1}^{\Delta(G)+1} C_t)|= 
    1 + k(\deg_{G}(v_j)) \quad \text{and} \quad
|C_{G\circ H}(u)\cap (\bigcup\limits_{t=1}^{\Delta(G)+1}C_t)|=k(\deg_{G}(v_i))$.
\end{center}

Since $k\geq 2$, we have $1 + k(\deg_{G}(v_j))\neq k(\deg_{G}(v_i))$. Thus,
$C_{G\circ H}(u)\neq C_{G\circ H}(v)$.

\textbf{Case (ii).} Suppose $u,v\in H_{v_i}$ where $i$ is odd. 
Let $B(H_{v_i},H_{v_{i+1}})$ be the bipartite graph with bipartitions $H_{v_i}$ and $H_{v_{i+1}}$. 
Let 
\[D_{B(H_{v_i},H_{v_{i+1}})}(x)=\{c(xy):xy\in E(B(H_{v_i},H_{v_{i+1}}))\}\]
be the set of colors on the incident edges of $x$ in $B(H_{v_i},H_{v_{i+1}})$. 
By Lemma \ref{Lemma 2.2}, we have $D_{B(H_{v_i},H_{v_{i+1}})}(u)\neq D_{B(H_{v_i},H_{v_{i+1}})}(v)$.
Moreover, $D_{B(H_{v_i},H_{v_{i+1}})}(u), D_{B(H_{v_i},H_{v_{i+1}})}(v)\subseteq C_{j}\cup C_{\Delta(G)+1}$ if $\theta(v_iv_{i+1})=j$. 
Finally, we have
$
C_{G\circ H}(u)\neq C_{G \circ H}(v)
$
since for $x\in\{u,v\}$,
\begin{center}
$
(C_{G\circ H}(x)\setminus D_{B(H_{v_i},H_{v_{i+1}})}(x))\cap(C_j\cup C_{\Delta(G)+1})=\emptyset.
$    
\end{center}

\textbf{Case (iii).} Let $u,v\in H_{v_i}$ where $i$ is even. Similar to Case (ii), we obtain 
$C_{G\circ H}(u)\neq C_{G\circ H}(v)$.
\end{proof}


\section{Skew Product}

\begin{thm}\label{Theorem 5.1}
    {\em Let $G$ be a Class 1 graph and $\chi''_{a}(H)\leq \Delta(H)+k$ for $k\in \{1,2,3\}$. Then $\chi''_{a}(G\Delta H)\leq \Delta(G\Delta H) +k$.}
\end{thm}

\begin{proof}
     Fix $k\in \{1,2,3\}$. The maximum degree of $G\Delta H$ is $\Delta(G\Delta H)=\Delta(H)+\Delta(G)\Delta(H)$. 
    We divide $\Delta(H) + \Delta(H)\Delta(G) + k$ colors into the following $\Delta(G) + 1$ disjoint sets:
\begin{align*}
C_0 &= \{a_1^0, a_2^0, \dots, a_{\Delta(H)+k}^0\}, \\
C_i &= \{a_1^i, a_2^i, \dots, a_{\Delta(H)}^i\}, 1\leq i\leq \Delta(G).
\end{align*}

\begin{figure}[ht]
    \centering
    \begin{tikzpicture}[scale=1.2]
        
        \draw[thick, rounded corners=12pt] (-3.5,-1.8) rectangle (-2.9,1.5);
        \node at (-3.2,1.8) {$P_n$};
        
        \node at (-3.2,1.2) {$\vdots$};
        \fill (-3.2,0.5)  circle (2pt) coordinate (P1);
        \fill (-3.2,-0.1) circle (2pt) coordinate (P2);
        \fill (-3.2,-0.7) circle (2pt) coordinate (P3);
        \fill (-3.2,-1.3) circle (2pt) coordinate (P4);
        
        \draw[thick] (P1) to[bend left=30] (P2);
        \draw[thick] (P2) to[bend right=30] (P3);
        \draw[thick] (P3) to[bend left=30] (P4);

        \draw[thick, rounded corners=12pt] (-1.1,-1.8) rectangle (-0.5,1.5);
        \node at (-0.8,1.8) {$(P_n)_{v_k}$};
        
        \node at (-0.8,1.2) {$\vdots$};
        \fill (-0.8,0.5)  circle (2pt) coordinate (L1);
        \fill (-0.8,-0.1) circle (2pt) coordinate (L2);
        \fill (-0.8,-0.7) circle (2pt) coordinate (L3);
        \fill (-0.8,-1.3) circle (2pt) coordinate (L4);
        
        \draw[thick] (L1) to[bend left=30] (L2);
        \draw[thick] (L2) to[bend right=30] (L3);
        \draw[thick] (L3) to[bend left=30] (L4);

        \draw[thick, rounded corners=12pt] (1.3,-1.8) rectangle (1.9,1.5);
        \node at (1.6,1.8) {$(P_n)_{v_\ell}$};
        
        \node at (1.6,1.2) {$\vdots$};
        \fill (1.6,0.5)  circle (2pt) coordinate (R1);
        \fill (1.6,-0.1) circle (2pt) coordinate (R2);
        \fill (1.6,-0.7) circle (2pt) coordinate (R3);
        \fill (1.6,-1.3) circle (2pt) coordinate (R4);
        
        \draw[thick] (R1) to[bend left=30] (R2);
        \draw[thick] (R2) to[bend right=30] (R3);
        \draw[thick] (R3) to[bend left=30] (R4);

        \draw[thick] (L1) -- (R2);
        \draw[thick] (L2) -- (R1);
        
        \draw[thick] (L2) -- (R3);
        \draw[thick] (L3) -- (R2);
        
        \draw[thick] (L3) -- (R4);
        \draw[thick] (L4) -- (R3);


        \begin{scope}[xshift=0.7cm, yshift=-2.7cm]

            \draw[thick, rounded corners=14pt] (-4.3,-0.6) rectangle (3.9,0.5);
            \node at (4.1,0) {$G$};

            \draw[rounded corners=10pt] (-3.8,-0.5) rectangle (3.4,0.4);

            \fill (-3.3,0) circle (1.5pt);
            \node[below=2pt] at (-3.3,0) {$v_1$};

            \node at (-2.5,0) {$\ldots$};

            \fill (-1.5,0) circle (1.5pt) coordinate (vk);
            \node[below=2pt] at (-1.5,0) {$v_k$};

            \fill (0.8,0) circle (1.5pt) coordinate (vl);
            \node[below=2pt] at (0.8,0) {$v_\ell$};

            \node at (1.8,0) {$\ldots$};

            \fill (2.8,0) circle (1.5pt);
            \node[below=2pt] at (2.8,0) {$v_m$};

            \draw[thick] (vk) -- (vl);

        \end{scope}

    \end{tikzpicture}
    \caption{Let $H=P_n$. Edges between $H_{v_k}=(P_n)_{v_k}$ and $H_{v_\ell}=(P_n)_{v_\ell}$ in $G \triangle P_n$ are displayed if $v_k v_\ell \in E(G)$.}
    \label{Figure 3}
\end{figure}
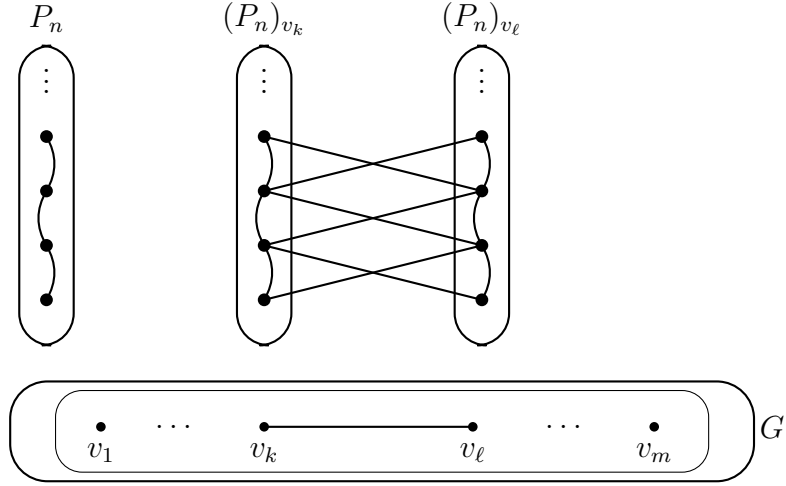

Let $H_{v_i}$ be the copy of the graph $H$ with respect to the vertex $v_i\in V(G)$.
We define a total coloring $c:V(G\Delta H)\cup E(G\Delta H)\rightarrow 
\bigcup\limits_{t=0}^{\Delta(G)} C_t$ of $G\Delta H$ as follows:
\begin{enumerate}
\item For each $v_r\in V(G)$, we consider an AVD total coloring of $H_{v_r}$ with colors from $C_0$ so that the corresponding elements of $H_{v_r}$ receive the same colors. In particular, let
$
\varphi:V(H)\cup E(H)\rightarrow C_0
$
be an AVD total coloring of \(H\). For every \(v_r\in V(G)\), define
$c((v_r,x))=\varphi(x)$ for all $x\in V(H)$,
and
$c((v_r,x)(v_r,y))=\varphi(xy)$ for all $xy\in E(H)$.

\item Since $G$ is a Class 1 graph, there exists a proper edge coloring $\theta:E(G)\rightarrow \{1,...,\Delta(G)\}$ of $G$.
For every edge $v_kv_l\in E(G)$ with $\theta(v_kv_l)=i$, 
let $\phi_{kl}: E(B(H_{v_k},H_{v_l})) \rightarrow C_i$ be a proper edge coloring of $B(H_{v_k},H_{v_l})$ with colors from $C_i$, where $B(H_{v_k},H_{v_l})$ is the bipartite graph with bipartitions $H_{v_k}$ and $H_{v_l}$ (see Fig. \ref{Figure 3} when $H$ is the path graph $P_n$). Such a coloring exists by Theorem \ref{Theorem 2.3} since
\begin{center}
$
\Delta(B(H_{v_k},H_{v_l}))
=
\Delta(H)
=
|C_i|.
$
\end{center}
We define $c(e)=\phi_{kl}(e)$ for every $e\in E(B(H_{v_k},H_{v_l}))$.
\end{enumerate}

\begin{claim}\label{Claim 4}
{\em The coloring $c$ is an AVD total coloring of $G\Delta H$.}
\end{claim}

\textit{Proof:}
By construction, $c$ is a proper total coloring of $G\Delta H$.
It suffices to show that for any $xy\in E(G\Delta H)$,
$
C_{G\Delta H}(x)\neq C_{G\Delta H}(y)
$. 
Let $xy\in E(G\Delta H)$.

\textbf{Case (i).} 
Let $x=(v_i,a)\in V(H_{v_i})$ and $y=(v_i,b)\in V(H_{v_i})$ for some $a,b\in V(H)$ and $v_{i} \in V(G)$.
Then $C_{H_{v_i}}(x)\neq C_{H_{v_i}}(y)$
as we considered an AVD-total coloring of $H_{v_i}$ in (1).
Since
$
C_{H_{v_i}}(x),\; C_{H_{v_i}}(y)\subseteq C_0,
$
and
\[
C_{G\Delta H}(x)\setminus C_{H_{v_i}}(x),\;
C_{G\Delta H}(y)\setminus C_{H_{v_i}}(y)
\subseteq \bigcup_{t=1}^{\Delta(G)} C_t,
\]
where
$
\left(\bigcup_{t=1}^{\Delta(G)} C_t\right)\cap C_0=\emptyset,
$
it follows that
$
C_{G\Delta H}(x)\neq C_{G\Delta H}(y).
$

\textbf{Case (ii).} 
Let $x=(v_i,a)\in V(H_{v_i})$ and $y=(v_j,b)\in V(H_{v_j})$ for some $a,b\in V(H)$ and $v_i,v_j\in V(G)$ such that $i\neq j$.
Let
$x'=(v_j,a)$
be the vertex of $H_{v_j}$ corresponding to $x$. Since the elements of every copy of $H$ is colored identically,
$
C_{H_{v_i}}(x)=C_{H_{v_j}}(x').
$
However, by the definition of skew product, $x'$ and $y$ are adjacent in $H_{v_j}$ and $H_{v_j}$ is AVD-total colored. Thus, we have
$
C_{H_{v_j}}(x')\neq C_{H_{v_j}}(y).
$
Hence,
$
C_{H_{v_i}}(x)\neq C_{H_{v_j}}(y).
$
Applying the same argument as in Case (i), we conclude that $C_{G\Delta H}(x)\neq C_{G\Delta H}(y)$.
\end{proof}

\section{Comb Product}

\begin{thm}\label{Theorem 6.1}
{\em Let $G$ and $H$ be two connected graphs such that $G$ satisfies the TCC and $H$ has AVD-Type 2. 
If $n\leq \deg_{H}(o)+1$ where $n$ is the order of $G$ and $o$ is the distinguished vertex of $G\triangleright_o H$, 
then $G\triangleright_o H$ satisfies the AVDTCC.}
\end{thm}

\begin{proof}
We note that $\Delta(G \triangleright_o H) =\max \{\Delta(G)+\Delta(H), \Delta(G)+\deg_H(o),\Delta(H)\} $. 
Let $H_1,...,H_n$ be the copies of $H$ in $G \triangleright_o H$ 
with respect to the vertices $v_1,...,v_n$ in $G$ (see Fig. \ref{Figure 4}). 

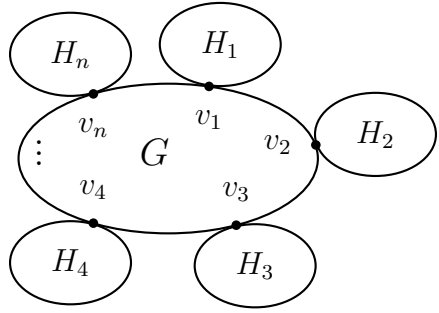
\begin{figure}[ht]
    \centering
    \begin{tikzpicture}[scale=0.9]
        \tikzset{
            h_ellipse/.style={
                draw,
                thick,
                ellipse,
                minimum width=1.6cm,
                minimum height=1.1cm,
                font=\normalsize,
                inner sep=0pt
            }
        }

        \draw[thick] (0,0) ellipse (2.2cm and 1.1cm);
        \node[font=\large] at (-0.2,0.1) {$G$};

        \coordinate (v1) at (0.6, 1.06);
        \fill (v1) circle (2pt);
        \node[below=5pt] at (v1) {$v_1$};

        \coordinate (v2) at (2.17, 0.2);
        \fill (v2) circle (2pt);
        \node[left=5pt] at (v2) {$v_2$};

        \coordinate (v3) at (1.0, -0.98);
        \fill (v3) circle (2pt);
        \node[above=5pt] at (v3) {$v_3$};

        \coordinate (v4) at (-1.1, -0.95);
        \fill (v4) circle (2pt);
        \node[above=5pt] at (v4) {$v_4$};

        \coordinate (vn) at (-1.1, 0.95);
        \fill (vn) circle (2pt);
        \node[below=5pt] at (vn) {$v_n$};

        \node[font=\large] at (-1.9, 0.2) {$\vdots$};

        \node[h_ellipse, anchor=255] at (v1) {$H_1$};

        \node[h_ellipse, anchor=190] at (v2) {$H_2$};

        \node[h_ellipse, anchor=115] at (v3) {$H_3$};

        \node[h_ellipse, anchor=60] at (v4) {$H_4$};

        \node[h_ellipse, anchor=300] at (vn) {$H_n$};

    \end{tikzpicture}
    \caption{The comb product $G \triangleright_o H$.}
    \label{Figure 4}
\end{figure}

We consider the following cases.

\textbf{Case (I).} $\Delta(G \triangleright_o H) = \Delta(G)+\Delta(H)$ or $\Delta(G \triangleright_o H) = \Delta(G)+\deg_H(o)$.

In this case, either $o$ is the vertex of $H$ of maximum degree or $o_1\ne o$ is the vertex of $H$ of maximum degree and $\deg_H(o_1)=\Delta(H)\leq \Delta(G)+\deg_H(o)$.
Let $A$ be a set of $\Delta(G)+2$ colors.
Since $G$ satisfies the TCC, let $f_G: V(G)\cup E(G)\rightarrow A$ be a proper total coloring of $G$.
Since $o\in H_i$ is merged with the $i^{th}$ vertex $v_i$ in $G$, vertex $o$ is colored with respect to $f_G$, say with color $a_i$. 
Since the number of edges incident to $v_i$ in $G$ is at most $\Delta(G)$, and $f_G$ uses $\Delta(G)+2$ colors from $A$, the set $A\setminus C_G(v_i)$ has at least one unused element, say $b_i$.

{\em Step 1:} 
We begin by coloring $H_i$ for each $1\leq i\leq n$.
Consider a set of $\deg_{H_i}(o)+1$ new colors, say $1,2,..., \deg_{H_i}(o),t$. 
Let $R_1 = \{1,...,\deg_{H_i}(o),t\}\backslash \{t\}$.
For each $2\le i\le n$, define
\[
\begin{aligned}
R_i &= \{1,...,\deg_{H_i}(o),t\}\backslash \{i-1\}.
\end{aligned}
\]
This is possible since $n\leq \deg_{H}(o)+1$. For every $1\leq i\leq n$, we color the edges incident to $o$ in $H_i$ with the colors of $R_i$.
We color the other elements of $H_i$ (excluding $o$ and edges in $H_i$ incident to $o$) with colors of $R_i$, $a_i$, $b_i$, and other $\Delta(H_i)+2-(2+\deg_{H_i}(o))$ colors from $A$ such that the coloring of $H_i$ is an AVD-total coloring. This is possible since 
$\Delta(H_i)-\deg_{H_i}(o)\leq \Delta(G)$, $|A|=\Delta(G)+2$,
$H_i$ is a copy of $H$ and $H$ has AVD-Type 2.

{\em Step 2:} We recolor the elements of $H_i$ for each $1\leq i\leq n$.
Let $\{r_1^i,..., r_{\deg_G(v_i)}^i\}$ be an enumeration of the open neighborhood of $v_i$ in $G$.
For each $i$, we pick the color of the edge $v_ir_1^i$, say $f_G(v_ir_1^i)$.
\renewcommand{\thefootnote}{\fnsymbol{footnote}}
If $H_i$ is colored with $f_{G}(v_ir_1^i),$ then we recolor all elements of $H_i$ that are colored with $f_G(v_ir_1^i)$ with the color $i-1$ 
for all $2\leq i\leq n$
and with the color $t$ if $i=1$.\footnotemark
\footnotetext{This is possible since ${i-1}\notin R_i$ for all $2\leq i\leq n$ and $t\not\in R_1$.}
\renewcommand{\thefootnote}{\arabic{footnote}}

Let $f_{H_i}$ be the total coloring of $H_i$ obtained after step 2.
Let
$D=\{1,2,\ldots,\deg_H(o),t\}$.
Then $|A\cup D|=\Delta(G)+2+\deg_H(o)+1$ and $A\cap D=\emptyset$.
We define a total coloring $f:V(G \triangleright_o H)\cup E(G \triangleright_o H)\rightarrow A\cup D$ of $G \triangleright_o H$ by
\[
f(x)=
\begin{cases}
f_G(x), & \text{if } x\in V(G)\cup E(G),\\
f_{H_i}(x), & \text{if } x\in \bigl(V(H_i)\setminus\{v_i\}\bigr)\cup E(H_i),\
1\le i\le n.
\end{cases}
\]

\begin{claim}\label{Claim 5}
{\em The coloring $f$ is an AVD-total coloring of $G \triangleright_o H$.}
\end{claim}

\textit{Proof:}
It is clear that $f$ is a proper total coloring of $G \triangleright_o H$.
Choose two vertices $u,v \in V(G \triangleright_o H)$ such that $u$ and $v$ are adjacent in $G \triangleright_o H$. 

\textbf{Case (i).} Let $u,v \in V(H_i)$ be such that $u,v \neq o$. 
It is easy to see that  $C_{G \triangleright_o H}(u)=C_{H_i}(u)$, $C_{G \triangleright_o H}(v)=C_{H_i}(v)$ and $C_{Hi}(u)\neq C_{H_i}(v)$. Thus, $C_{G \triangleright_o H}(u)\neq C_{G \triangleright_o H}(v)$. 

\textbf{Case (ii).} Let $u=o \in H_i$ and $v\in H_i$. Then $C_{G \triangleright_o H}(v)=C_{H_i}(v) \neq C_{H_i}(o)$.
Let $\{r_1^i,..., r_{\deg_G(o)}^i\}$ be the enumeration of the open neighborhood of $o=v_i$ in $G$ as in step 2.
By recoloring of $H_i$, the color $f(or_1^i)$ does not appear in the elements of $H_i$, and hence does not appear in the incident edges of $v$ in $H_i$.
Thus, $C_{G \triangleright_o H}(o)=C_{H_i}(o)\cup D_G(o)$, where
\[
C_{H_i}(o)\cap D_G(o)=\emptyset,\qquad
f(or_1^i)\in D_G(o),\qquad
f(or_1^i)\notin C_{G \triangleright_o H}(v).
\]
Thus, $f(or_1^i)\in C_{G \triangleright_o H}(o)$, and so $C_{G \triangleright_o H}(o)\neq C_{G \triangleright_o H}(v)$.

\textbf{Case (iii).} Let $u=v_i \in V(G)$ and $v=v_j \in V(G)$, where $i\neq j$.
In the coloring of the elements of $H_i,$ we colored the edges incident to $v_i$ and $v_j$ (in $H_i$ and $H_j$) with the colors of $R_i$ and $R_j$, respectively.
We note that $R_i\neq R_j$.
Also, the colors used in $R_i$ and $R_j$ does not appear in the edges incident to $u,v$ in $G$. 
Consequently, $C_{G \triangleright_o H}(u)\neq C_{G \triangleright_o H}(v)$.

This concludes the proof of Claim \ref{Claim 5}.

\textbf{Case (II).} 
$\Delta(G \triangleright_o H)=\Delta(H).$

In this case, we have $\Delta(H)+3$ colors available to color $G \triangleright_o H$. We note that $\Delta(H)\geq \Delta(G)+\deg_H(o)$.
Among $\Delta(H)+3$ available colors, let $Y$ be a set of $\Delta(G)+\deg_H(o)+3$ colors and $Z$ be the set of all available colors that are not in $Y$.
Let $A\subseteq Y$ be a set of $\Delta(G)+2$ colors and let $D=\{1,...,\deg_{H}(o), t\}$ be a set of colors of $Y\backslash A$.
As in Case (I), let $f_{G}:V(G)\cup E(G)\rightarrow A$ be a proper total coloring of $G$ and we color the edges incident to $o$ in $H_i$ with the colors of $R_i$ for each $1\leq i\leq n$. 
Assume $a_i, b_i$ as in Case (I).
We color the other elements of $H_i$ (excluding the distinguished vertex $o\in V(H_i)$ and edges incident to $o$ in $H_i$) with colors of $R_i,a_i,b_i$,   $(\Delta(G)+\deg_H(o)+2)-(2+\deg_H(o))$ colors from the available colors in $A$ and extra colors form $Z$ such that the coloring of $H_i$ is AVD-total coloring. 
Recolor $H_i$ as in step (2) of Case (I) and define $f_{H_i}$ and $f$ as in Case (I). Similar to the arguments of Claim \ref{Claim 5}, $f$ is an AVD-total coloring of $G \triangleright_o H$.
\end{proof}

\begin{thm}\label{Theorem 6.2}
{\em Let $G$ and $H$ be two connected graphs such that $G$ and $H$ have AVD-Type 2. Then $G\triangleright_o H$ satisfies the AVDTCC.}
\end{thm}

\begin{proof}[Proof Sketch]
Let $n=|V(G)|$ and $d=\deg_H(o)$. 
Let $A$ be a set of $\Delta(G)+2$ colors disjoint from $\{1,\ldots,d,t\}$.
Since $G$ has AVD-Type 2, it satisfies the TCC.
If $n\leq d+1$, the result follows directly from Theorem~\ref{Theorem 6.1}.
Suppose that $n>d+1$. In the proof of Theorem~\ref{Theorem 6.1}, retain the same construction, except for the definition of $f_G$ and the choice of the sets $R_i$.
Since $G$ has AVD-Type 2, let $f_G:V(G)\cup E(G)\rightarrow A$ be an AVD-total coloring of $G$. 
We use the $d+1$ colors
$\{1,2,\ldots,d,t\}$
and define
$R_1=\{1,2,\ldots,d\}$,
and, for $2\leq i\leq d+1$, let
$R_i=\{1,2,\ldots,d,t\}\setminus\{i-1\}$.
For $d+2\leq i\leq n$, define
$R_i=R_{d+1}$.
Thus, $R_1,\ldots,R_{d+1}$ are pairwise distinct, while
\[
R_{d+2}=\cdots=R_n=R_{d+1}.
\]
All parts of the proof of Theorem~\ref{Theorem 6.1} remain unchanged, except for Case (iii) of Claim~\ref{Claim 5}. Let $u=v_i \in V(G)$ and $v=v_j \in V(G)$ where $i\neq j$ as in Case (iii) of Claim~\ref{Claim 5}. 
Since $f_G$ is an AVD-total coloring of $G$, we have $C_G(v_i)\neq C_G(v_j)$ for each $v_iv_j\in E(G)$.
As the colors used for $f_G$ are chosen disjoint from $\{1,2,...,d,t\}$, we have
$
C_{G\triangleright_o H}(v_i)\neq
C_{G\triangleright_o H}(v_j).
$
\end{proof}

\section{Cover Product}

\begin{thm}\label{Theorem 7.1}
    {\em Let $G$ and $H$ be two total colorable connected graphs. Let $k_{1}$ and $k_{2}$ be the vertex covering
numbers of $G$ and $H$, respectively. If either 
\begin{center}
    $\Delta(H) \leq \Delta(G)$ and $k_{1}+1 < k_{2}$ or $\Delta(G) \leq \Delta(H)$ and
$k_{2}+1 < k_{1}$ 
\end{center}
then $G \oast H$ satisfies AVDTCC.}
\end{thm}

\begin{proof}
Let $V(VC(G))=\{v_1,...,v_{k_1}\}$ and $V(VC(H))=\{u_1,...,u_{k_2}\}$ be the minimum vertex covers of $G$ and $H$ respectively.

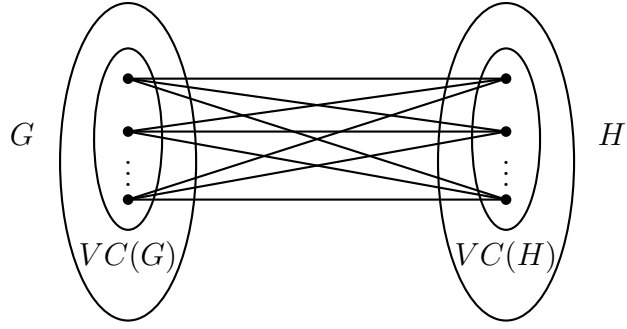
\begin{figure}[ht]
    \centering
    \begin{tikzpicture}[scale=1]

        \draw[thick] (-2.5, 0.2) ellipse (0.9cm and 2.1cm);
        \node[below=4pt] at (-3.9, 1) {$G$};
        
        \draw[thick] (-2.5, 0.5) ellipse (0.45cm and 1.2cm);
        \node[below] at (-2.5, -0.7) {$VC(G)$};
        
        \fill (-2.5, 1.3)  circle (2pt) coordinate (L1);
        \fill (-2.5, 0.6)  circle (2pt) coordinate (L2);
        \fill (-2.5, -0.3) circle (2pt) coordinate (L3);
        \node at (-2.5, 0.15) {$\vdots$};

        \draw[thick] (2.5, 0.2) ellipse (0.9cm and 2.1cm);
        \node[below=4pt] at (3.9, 1) {$H$};
        
        \draw[thick] (2.5, 0.5) ellipse (0.45cm and 1.2cm);
        \node[below] at (2.5, -0.7) {$VC(H)$};
        
        \fill (2.5, 1.3)  circle (2pt) coordinate (R1);
        \fill (2.5, 0.6)  circle (2pt) coordinate (R2);
        \fill (2.5, -0.3) circle (2pt) coordinate (R3);
        \node at (2.5, 0.15) {$\vdots$};

        \draw[thick] (L1) -- (R1);
        \draw[thick] (L1) -- (R2);
        \draw[thick] (L1) -- (R3);
        
        \draw[thick] (L2) -- (R1);
        \draw[thick] (L2) -- (R2);
        \draw[thick] (L2) -- (R3);
        
        \draw[thick] (L3) -- (R1);
        \draw[thick] (L3) -- (R2);
        \draw[thick] (L3) -- (R3);

    \end{tikzpicture}
    \caption{The cover product $G \circledast H$.}
    \label{Figure 5}
\end{figure}

\begin{claim}\label{Claim 6}
{\em The vertices of maximum degree of $G$ and $H$ are in $VC(G)$ and $VC(H)$.}
\end{claim}

\textit{Proof:}
Pick any vertex $x \in V(G) \setminus VC(G)$. Since $VC(G)$ is a vertex cover of $G$, every neighbor of $x$ in $G$ belongs to $VC(G)$, and hence $\deg_G(x) \le |VC(G)|$. Moreover, by the definition of $G \circledast H$, $x$ is adjacent to no vertices outside $VC(G)$. Thus $\deg_{G \circledast H}(x) \le |VC(G)|$.
However, any vertex $v\in V(VC(H))$ is adjacent to $|VC(G)|+\deg_{H}(v)>|VC(G)|$ vertices in $G\oast H$ (see Fig. \ref{Figure 5}). This is possible since $\deg_{H}(v)>0$ as $H$ is connected.
So, no vertex in $V(G)\setminus VC(G)$ can be a vertex of
maximum degree in $G\oast H$. 
Similarly, every vertex of maximum degree
of $H$ belongs to $VC(H)$.
This concludes the proof of Claim \ref{Claim 6}.

By Claim \ref{Claim 6}, $\Delta(G\oast H)= max \{\Delta(G)+k_2, \Delta(H)+k_1\}$.
We assume $\Delta(H)\leq \Delta(G)$ and $k_1+1<k_2$. 
The other case will be analogous.
Let $D=\{k_2+2,...,k_2+\Delta(G)+3\}$ be a set of colors with $|D| = \Delta(G) + 2$. 
Since $G$ and $H$ satisfy the TCC, there are proper total colorings $f_1 : V(G)\cup E(G) \to D$ and $f_2 : V(H)\cup E(H) \to D$ of $G$ and $H$ respectively. 
Let $C=\{1,...,k_2+1\}$. 
We recolor the vertices of $VC(H)$ and color the join edges between
$VC(G)$ and $VC(H)$ using the colors in $C$. Consider a Latin square $M=[m_{i,j}]$ of order $k_2+1$ 
as in Lemma \ref{Lemma 2.2}. 
Let $M'=[m_{i,j}]_{1\leq i\leq k_2,1\leq j\leq k_1+1}$ be the submatrix consisting of the first $k_2$ rows and first $k_1+1$ columns.
Consider the complete bipartite graph $B$ with bipartitions $VC(G)$ and $VC(H)$, which is a subgraph of $G \oast H$. We define
\begin{align*}
g(x)=
\begin{cases}
m_{i,j} & \text{if } x=u_iv_j \text{ with }  1\leq i \leq k_2, 1\leq j \leq k_1 \\[2mm]
m_{i, k_1+1} & \text{if } x=u_i \text{ with } 1\leq i \leq k_2.
\end{cases}
\end{align*}
We now define a total coloring $f:V(G\oast H)\cup E(G\oast H)\rightarrow C\cup D$ by
\[
f(x)=
\begin{cases}
f_1(x), & \text{if } x\in V(G)\cup E(G),\\
f_2(x), & \text{if } x\in E(H)\cup V(H)\backslash V(VC(H)),\\
g(x), & \text{if } x\in E(B)\cup V(VC(H)).
\end{cases}
\]
It is clear that $C\cap D=\emptyset$ and $|C\cup D|=\Delta(G)+k_2+3$.

\begin{claim}\label{Claim 7}
{\em The coloring $f:G\oast H\rightarrow C\cup D$ is an AVD-total coloring of $G\oast H$.}
\end{claim}

\textit{Proof:}
By construction, 
$f$ is a proper total coloring of $G\oast H$.
Pick $uv\in E(G\oast H)$. 

\textbf{Case (i).} Let $u\in V(H)\backslash V(VC(H))$ and $v\in V(VC(H))$. 
Then $C_{G\oast H}(u)\neq C_{G\oast H}(v)$ since
\begin{center}
    $C_{G\oast H}(u)=C_{H}(u)\subseteq D$ and $C_{G\oast H}(v)\cap C\neq \emptyset$. 
\end{center}
Similarly, if $u\in V(G)\backslash V(VC(G))$ and $v\in V(VC(G))$ then $C_{G\oast H}(u)\neq C_{G\oast H}(v)$.

\textbf{Case (ii).} Suppose $u,v\in V(VC(H))$.
Then $C_{G\oast H}(x)=C_{B}(x)\cup D_{H}(x)$ for $x\in \{u,v\}$
where 
\begin{center}
$C_{B}(u)\neq C_{B}(v)$    
\end{center}
by Lemma \ref{Lemma 2.2} and $D_{H}(u), D_{H}(v)\subseteq D$. Therefore, $C_{G\oast H}(u)\neq C_{G\oast H}(v)$. 
Similarly, if $u,v\in V(VC(G))$ then 
$C_{G\oast H}(u)\neq C_{G\oast H}(v)$ by Lemma \ref{Lemma 2.2}.

\textbf{Case (iii).} $u\in V(VC(G))$ and $v\in V(VC(H))$. 
We note that $C_{B}(u)\neq C_{B}(v)$ since 
\begin{center}
    $|C_{B}(v)\cap C|=k_1+1<k_2=|C_{B}(u)\cap C|$. 
\end{center}
Moreover, 
$C_{G\oast H}(v)=C_{B}(v)\cup D_{H}(v)$ and
$C_{G\oast H}(u)=C_{B}(u)\cup D_{G}(u)$ where $D_H(v), D_{G}(u)\subseteq D$.
Consequently, $C_{G\oast H}(u)\neq C_{G\oast H}(v)$.

This shows that $C_{G\oast H}(x)\neq C_{G\oast H}(y)$ for every edge $xy\in E(G\oast H)$. 

Thus, $\chi''_a(G\oast H) \le \Delta(G\oast H)+3$, and therefore $G\oast H$ satisfies the AVD-TCC.
\end{proof}

\section{Indu--Bala Product}

\begin{thm}\label{Theorem 8.1}
    {\em Let $G_1$ and $G_2$ be graphs of order $m$ and $n$ respectively such that
    \begin{enumerate}
        \item $m>n+1$,
        \item $\Delta(G_2)\geq\Delta(G_1)$, and
        \item $G_2$ is a total colorable graph.
    \end{enumerate}
    Then, $G=G_1\blacktriangledown G_2$ satisfies the AVDTCC.    
    }
\end{thm}

\begin{proof}
Let $H_1$ be the first copy of $G_1 \vee G_2$ and $H_2$ be the second copy of $G_1 \vee G_2$. Let $G_1'$ and $G_2'$ denote the copies of $G_1$ and $G_2$ in $H_2$. 
Suppose $V(G_1)=\{u_1,...,u_m\}$, $V(G_2)=\{v_1,...,v_n\}$, $V(G'_1)=\{u'_1,...,u'_m\}$ and $V(G'_2)=\{v'_1,...,v'_n\}$ be the set of vertices of $G_1, G_2, G'_1$ and $G'_2$ respectively (see Fig. \ref{Figure 6}). 
We note that $\Delta(G)=\max\{\Delta(G_1)+n,\ \Delta(G_2)+m+1\}=\Delta(G_2)+m+1.$

\begin{figure}[ht]
    \centering
    \begin{tikzpicture}[scale=1.7]

        \draw[thick, dashed, gray!80, rounded corners=18pt]
            (-4.4,-1.5) rectangle (-0.6,1.2);

        \draw[thick, dashed, gray!80, rounded corners=18pt]
            (0.6,-1.5) rectangle (4.4,1.2);

        \draw[thick, rounded corners=12pt]
            (-4.1,-1) rectangle (-3.1,1.05);
        \node[below=3pt] at (-3.6,-1) {$G_1$};

        \node[below=3pt] at (-2.3,1.2) {$H_1=G_1\vee G_2$};

        \draw[thick, rounded corners=12pt]
            (-1.5,-0.75) rectangle (-0.9,1.05);
        \node[below=3pt] at (-1.2,-0.8) {$G_2$};

        \draw[thick, rounded corners=12pt]
            (0.9,-0.75) rectangle (1.5,1.05);
        \node[below=3pt] at (1.2,-0.8) {$G_2'$};
        \node[below=3pt] at (2.3,1.2) {$H_2=G'_1\vee G'_2$};

        \draw[thick, rounded corners=12pt]
            (3.1,-1) rectangle (4.1,1.05);
        \node[below=3pt] at (3.6,-1) {$G_1'$};

        \fill (-3.6,0.7) circle (2pt) coordinate (U1);
        \fill (-3.6,-0.7) circle (2pt) coordinate (Um);
        \node[left=3pt] at (U1) {$u_1$};
        \node[left=3pt] at (Um) {$u_m$};
        \node at (-3.6,0) {$\vdots$};

        \fill (-1.2,0.55) circle (2pt) coordinate (V1);
        \fill (-1.2,-0.25) circle (2pt) coordinate (Vn);
        \node[above=2pt] at (V1) {$v_1$};
        \node[below=2pt] at (Vn) {$v_n$};
        \node at (-1.2,0.25) {$\vdots$};

        \fill (1.2,0.55) circle (2pt) coordinate (V1p);
        \fill (1.2,-0.25) circle (2pt) coordinate (Vnp);
        \node[above=2pt] at (V1p) {$v_1'$};
        \node[below=2pt] at (Vnp) {$v_n'$};
        \node at (1.2,0.25) {$\vdots$};

        \fill (3.6,0.7) circle (2pt) coordinate (U1p);
        \fill (3.6,-0.7) circle (2pt) coordinate (Ump);
        \node[right=3pt] at (U1p) {$u_1'$};
        \node[right=3pt] at (Ump) {$u_m'$};
        \node at (3.6,0) {$\vdots$};

        \draw[thick] (U1) -- (V1);
        \draw[thick] (U1) -- (Vn);
        \draw[thick] (Um) -- (V1);
        \draw[thick] (Um) -- (Vn);

        \draw[thick] (V1) -- (V1p);
        \draw[thick] (Vn) -- (Vnp);

        \draw[thick] (V1p) -- (U1p);
        \draw[thick] (V1p) -- (Ump);
        \draw[thick] (Vnp) -- (U1p);
        \draw[thick] (Vnp) -- (Ump);

    \end{tikzpicture}
    \caption{The Indu-Bala product $G_1 \blacktriangledown G_2$.}
    \label{Figure 6}
\end{figure}
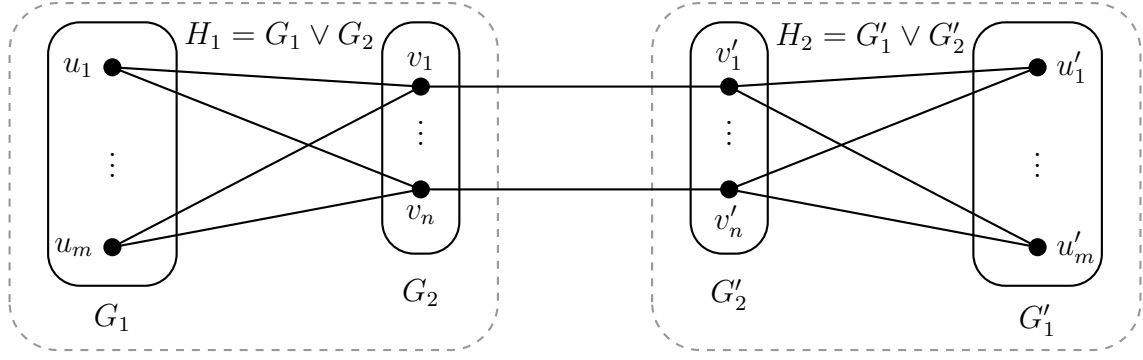

Consider a Latin square $M=[m_{i,j}]$ of order $t=m+1$ as in Lemma \ref{Lemma 2.2}. 
Let $M'$ consists of the first $m$ rows and $n+1$ columns of $M$.
We define
    \[
    g_1(x)=
    \begin{cases}
    m_{i,j}, & x=u_i v_j, 1\leq i\leq m, 1\leq j\leq n\\[2mm]
    m_{i,n+1}, & x=u_i, 1\leq i\leq m.
    \end{cases}
    \]
Consider the complete bipartite graph $B(H_1)$ with bipartitions $V(G_1)$ and $V(G_2)$, which is a subgraph of $H_1$. Let $C=\{1,...,m+1\}$ be the set of entries of $M$. 

\begin{claim}\label{Claim 8}
{\em The following holds:
\begin{enumerate}
    \item No two adjacent edges in $B(H_1)$ receive the same color. 
    \item Each vertex $u_i$ receives a color distinct from its incident edges in $B(H_1)$. 
    \item 
    $D_{B(H_1)}(v_i)\neq D_{B(H_1)}(v_j)$
    and
    $C_{B(H_1)}(u_i)\neq C_{B(H_1)}(u_j)$ for any  $i\neq j$.
\end{enumerate}
}
\end{claim}

\textit{Proof:}
(1) and (2) hold since each element of $M$
occurs exactly once in each row and exactly once in each column of $M$ and (3) follows from Lemma \ref{Lemma 2.2}.

Let $D=\{m+2, m+3,...,m+\Delta(G_{2})+3\}$ be a set of $\Delta(G_2)+2$ colors.
Since $G_2$ satisfies the TCC, there exists a proper total coloring $g_2:V(G_2)\cup E(G_2)\rightarrow D$ of $G_2$.
Let $g_3:E(G_1)\rightarrow D$ be a proper edge coloring of $G_1$. This is possible since $\chi'(G_1)\leq \Delta(G_1)+1<\Delta(G_2)+2$.\\
We define a total coloring
$g:V(H_1)\cup E(H_1)\rightarrow C\cup D$ by
\[
g(x)=
\begin{cases}
g_1(x), & \text{if } x\in V(G_1)\cup E(B(H_1)),\\
g_2(x), & \text{if } x\in V(G_2)\cup E(G_2),\\
g_3(x), & \text{if } x\in E(G_1).
\end{cases}
\]
Now, we color the elements of $H_2$.
Define $g_1'$ using the elements of $M'$ as follows:
\[
g_1'(x)=
\begin{cases}
m_{i,j+1}, & x=u_i'v_j', 1\leq i\leq m, 1\leq j\leq n\\[2mm]
m_{i,n+2}, & x=u_i', 1\leq i\leq m.
\end{cases}
\]
Let $\sigma$ be a cyclic permutation of all the colors in $D$, so that $\sigma(d)\neq d$ for any $d\in D$. 
Let $g_2':V(G_2')\cup E(G_2')\rightarrow D$ be the total coloring of $G_2'$  defined by 
\[
g_2'(x)=
\begin{cases}
\sigma(g_2(v_i)), & \text{if }x=v'_i,\\[1mm]
\sigma(g_2(v_iv_j)), & \text{if }x=v_i'v_j'.
\end{cases}
\]
The definition of $g_2'$ ensures that $g_2'(v'_i)\neq g_2(v_i)$ for each $v_i\in V(G_2)$ and $v'_i\in V(G'_2)$ where $1\leq i\leq n$.
Similar to the coloring $g_3$, let $g_3': E(G_1')\rightarrow D$ be a proper edge coloring  of $G_1'$. 
We define a total coloring
$g':V(H_2)\cup E(H_2)\rightarrow C\cup D$ by
\[
g'(x)=
\begin{cases}
g'_1(x), & \text{if } x\in V(G'_1)\cup E(B(H_2)),\\
g'_2(x), & \text{if } x\in V(G'_2)\cup E(G'_2),\\
g'_3(x), & \text{if } x\in E(G'_1).
\end{cases}
\]
We use a new color $c$ to color all the edges $\{v_iv_i':1\le i\leq n\}$.\\
Define a total coloring $h : E(G)\cup V(G)\rightarrow C \cup D \cup \{c\}$ of $G$ by
\[h(x)=
\begin{cases}
g(x), & \text{if } x \text{ is an element of } H_1,\\[1ex]
g'(x), & \text{if } x \text{ is an element of } H_2,\\[1ex]
c, & \text{if } x \in \{v_iv_i' : 1 \leq i \leq n\,\}.
\end{cases}
\]
It is clear that \(h\) is a proper total coloring of $G$ and $|C\cup D\cup \{c\}|=\left(\Delta(G_2)+m+1\right)+3$.

\begin{claim}\label{Claim 9}
{\em $D_{B(H_1)}(v_j)\neq D_{B(H_2)}(v_j')$ for each $1\leq j\leq n$.}
\end{claim}

\textit{Proof:}
The definition of $g_1'$ ensures that if $B(H_i)$ is the bipartite induced subgraph of $H_i$ $(i=1,2)$, then $D_{B(H_1)}(v_j)=\{m_{1,j},...,m_{m,j}\}\neq \{m_{1,j+1},...,m_{m,j+1}\}= D_{B(H_2)}(v_j')$ by Lemma \ref{Lemma 2.2}.    

\begin{claim}\label{Claim 10}
{\em The coloring $h$ is an AVD-total coloring of $G$.}
\end{claim}

\textit{Proof:}
Let $uv \in E(G)$ be an arbitrary edge. 

\textbf{Case (i):}
Suppose $u=u_i$ and $v=u_j$ such that $uv\in E(G_1)$.
Then for $x\in \{i,j\}$,
\begin{center}
    $C_G(u_x)=C_{B(H_1)}(u_x)\cup D_{G_1}(u_x)$,
\end{center}
where $C_{B(H_1)}(u_x)\subseteq C$ and $D_{G_1}(u_x)\subseteq D$. Moreover, $C_{B(H_1)}(u_i)\neq C_{B(H_1)}(u_j)$ by Lemma \ref{Lemma 2.2}.
Since $C\cap D=\emptyset$, we obtain $C_{G}(u)=C_G(u_i)\neq C_G(u_j)=C_{G}(v)$.
Similarly, if $uv\in E(G_1')$ then $C_G(u)\neq C_G(v)$.

\textbf{Case (ii):} 
Suppose $u=u_i$ and $v=v_j$ such that $uv\in E(H_1)$. 
Then $|C_G(u_i)\cap C|=n+1$ and
$|C_G(v_j)\cap C|=m$.
Since $m>n+1$, we obtain
$C_G(u_i)\neq C_G(v_j)$.
Similarly, if $u=u'_i$ and $v=v'_j$ such that $u'_iv'_j\in E(H_2)$, then $C_G(u'_i)\neq C_{G}(v'_j)$.     

\textbf{Case (iii):} 
Let $u=v_i, v=v_j$ such that $v_iv_j\in E(G_2)$. The following holds:
\begin{enumerate}
    \item $D_{B(H_1)}(v_i)\neq D_{B(H_1)}(v_j)$ by Lemma \ref{Lemma 2.2},
    
    \item $C_{G}(v_x)=D_{B(H_1)}(v_x)\cup h(v_xv_x')\cup C_{G_2}(v_x)$ for $x\in \{i,j\}$, 
    
    \item $C 
    \cap
    \left\{
    h(v_xv_x')\cup C_{G_2}(v_x)
    \right\}
    =\emptyset$ and $D_{B(H_1)}(v_x)\subseteq C$ for $x\in \{i,j\}$.
\end{enumerate}
Thus, $C_{G}(v_i)\neq C_{G}(v_j)$.
Similarly, if $u=v'_i$, $v=v'_j$, and $uv\in E(G'_2)$ then $C_{G}(v'_i)\neq C_{G}(v'_j)$.

\textbf{Case (iv):}
Let $u=v_i$, $v=v_i'$ and $v_iv'_i\in E(G)$.
Then, the following holds:
\begin{enumerate}
    \item $D_{B(H_1)}(v_i) \neq D_{B(H_2)} (v'_i)$ by Claim \ref{Claim 9},
    
    \item $C_G(v_i)=D_{B(H_1)}(v_i)\cup h(v_i v_i') \cup C_{G_2}(v_i)$,
    
    \item $C_G(v'_i)=D_{B(H_2)}(v'_i)\cup h(v_i v_i') \cup C_{G'_2}(v'_i)$,
    
    \item $D_{B(H_1)}(v_i), D_{B(H_2)}(v'_i)\subseteq C$,

    \item 
    $C \cap \left\{ h(v_i v_i') \cup C_{G_2}(v_i) \right\} = \emptyset$ and $C \cap \left\{ h(v_i v_i') \cup C_{G'_2}(v'_i) \right\} = \emptyset$.
\end{enumerate}

Consequently, $C_G(v_i) \neq C_G(v_i')$. 
\end{proof}

\textbf{Declarations}

\textbf{Conflict of interest:} The authors declare that they have no conflict of interest.\\

\textbf{Acknowledgment:} The first author thanks Amrita Vishwa Vidyapeetham, Coimbatore, India for the hospitality during his research visit. 

\end{document}